\documentclass[11pt]{article}
\usepackage{a4wide}
\usepackage{amssymb,amsmath,amsthm}
\providecommand{\abs}[1]{\lvert #1 \rvert}
\newcommand{\E}{\mathbb{E} \,}
\newcommand{\R}{\mathbb{R}}
\newcommand{\N}{\mathbb{N}}

\newcommand{\e}{\mathrm{e}}
\newtheorem{thm}{Theorem}
\newtheorem{lem}[thm]{Lemma}
\theoremstyle{definition}
\newtheorem{Def}[thm]{Definition}
\theoremstyle{remark}
\newtheorem*{Rem}{Remark}
\title{The symmetric property~($\tau$) for the Gaussian measure}
\author{Joseph Lehec
\footnote{
LAMA (UMR CNRS 8050)
Universit\'e Paris-Est.
}}
\date{December 2007}
\begin{document}
\maketitle
\begin{abstract}
We give a proof, based on the Poincar\'e inequality, of the symmetric property~($\tau$) for the Gaussian measure. If $f\colon\R^d\to\R$ is continuous, bounded from below and even, we define $Hf(x)=\inf_{y}  f(x+y)+ \frac{1}{2} |y|^2 $ and we have
\[ \int \e^{-f} \, d\gamma_d  \int \e^{Hf} \, d\gamma_d  \leq  1. \]
This property is equivalent to a certain functional form of the Blaschke-Santal\'o inequality, as explained in a paper by Artstein, Klartag and Milman.
\bigskip

\noindent
Published in Ann. Fac. Sci. Toulouse SŽr. 6, 17 (2) (2008) 357--370.
\end{abstract}
\section{Introduction}
The Blaschke-Santal\'o inequality states that if $K$ is a symmetric convex body of $\R^d$ then
\begin{equation}
\label{classique-santalo}
 \lvert K\rvert_d  \lvert  K^\circ \rvert_d \leq  \lvert D\rvert_d \lvert D^\circ\rvert_d = v_d^2 , 
\end{equation}
where $\lvert \cdot\rvert_d$ stands for the volume, $K^\circ$ is the polar body of $K$, $D$ the Euclidean ball and $v_d$ its volume. It was first proved by Blaschke for dimension $2$ and $3$ and Santal\'o \cite{santalo} extended the result to any dimension. K.~Ball in \cite{these-ball} was the first to prove a functional version of this inequality. Since then, several authors found either improvements or other versions of Ball's inequality, for example Lutwak and Zhang \cite{lz}, Artstein, Klartag and Milman \cite{akm}, or Fradelizi and Meyer \cite{fm}. \\
Ball's inequality implies in particular the following:
if $F$ is a non negative even measurable function on $\R^d$ then
\begin{equation}
\label{ball-fs} 
\int F(x) \, dx  
\int F^\circ (x) \, dx   \leq  (2\pi)^d ,
\end{equation}
where $F^\circ$ is the polar function of $F$:
\[ F^\circ (x) = \inf_{y\in \R^d}   \frac{ \e^{-x\cdot y} }{ F(y) } .\] 
Let $K$ be a symmetric convex body and $N_K$ the associated norm. Defining $F_K =\e^{ - \frac{1}{2} N_K^2  } $, it is easy to see that $(F_K)^\circ = F_{K^\circ}$. Besides, a standard computation shows that $v_d \int F_K = (2\pi)^{d/2} |K|_d$ and similarly for $K^\circ$. Therefore, applying \eqref{ball-fs} to $F_K$, we get \eqref{classique-santalo}. So \eqref{ball-fs} is a functional form of Santal\'o's inequality. \\
In \cite{akm}, Artstein, Klartag and Milman extend this inequality to the non even setting: they prove that \eqref{ball-fs} holds as soon as the barycenter of $F^\circ$ is at $0$ (which is true when $F$ is even). At the end of their paper is an interesting remark on the property~($\tau$). This property was introduced by Talagrand and named after him by Maurey in \cite{propriete-tau}. Here is the definition: let $\mu$  be a probability measure on $\R^d$ and $w$ a weight (which is a non negative function equal to $0$ at $0$). We say that $(\mu,w)$ satisfies property~($\tau$) if for any non negative continuous function $f$, the following inequality holds:
\begin{equation}
\label{tau}
 \int \e^{-f} \, d\mu  
   \int \e^{ f \boxdot w } \, d\mu  \leq  1 , 
\end{equation}
where $f\boxdot w$ is the infimum convolution of $f$ and $w$:
\[ f\boxdot w (x) = \inf_{y\in \R^d} \bigl( f(x+y) + w(y) \bigr) . \]
Clearly, we may replace ``$f\geq 0$'' by ``$f$ bounded from below''. Of course $f\boxdot w (x) \leq f(x)$ (take $y=0$). In order to find $f\boxdot w (x)$, one is allowed to move from $x$ to some $x+y$ such that $f(x+y)$ is smaller than $f(x)$, but the displacement costs $w(y)$. \\
Let $\gamma_d$ be the standard Gaussian measure on $\R^d$. For $x\in\R^d$, we let $\abs{x}$ be its Euclidean norm and we define the quadratic cost $c\colon x \mapsto \abs{x}^2$. Maurey remarked that $(\gamma_d,\tfrac{c}{4} )$ satisfies property~($\tau$). It follows from the famous Pr\'ekopa-Leindler inequality: let $u$,$v$ and $w$ be measurable functions satisfying $w(\frac{x+y}{2})\leq \tfrac{1}{2}(u(x)+v(y))$ for all $x$ and $y$ in $\R^d$, then 
\begin{equation}
\label{prekopa-leindler}
 \Bigl(  \int \e^{-u}  \, dx \Bigr)^\frac{1}{2}  \Bigl( 
              \int \e^{-v} \, dx  \Bigr)^\frac{1}{2}
          \leq    \int \e^{-w}  \, dx  .
\end{equation}
This is a reverse form of H\"older's inequality, we refer to \cite{ball} for a proof and selected applications. Applying this inequality to $u=f+\frac{c}{2}$, $v=-(f\boxdot \frac{c}{4})+\frac{c}{2}$ and $w=\frac{c}{2}$ yields property~($\tau$) for the Gaussian measure. \\
It is pointed out in \cite{akm} that there is a strong connection between property~($\tau$) and the functional version of Santal\'o's inequality. Indeed, applying \eqref{ball-fs} to $F=\e^{ -f-\frac{c}{2} }$ we obtain easily
\[ \int \e^{-f} \, d\gamma_d  
   \int \e^{ f \boxdot \tfrac{ c }{ 2 } } \,
   d\gamma_d   \leq  1 . \]
Hence, if we restrict to even functions we have property~($\tau$) for $(\gamma_d,\frac{c}{2} )$ (we gain a factor $2$ in the cost). This is what we call the \emph{symmetric} property~($\tau$), and our purpose is to prove it directly. We will show that it is related to the eigenvalues of the Laplacian in Gauss space. This will provide a new proof of the functional Santal\'o inequality, and this proof avoids using the usual Santal\'o inequality (for convex sets), which was not the case of Ball's proof. \\
From now on we denote by $H f$ the function
$ f \boxdot \frac{c}{2}$:
\[ H  f (x) = \inf_{y\in \R^d} \bigl( f(x+y) + \tfrac{1}{2} \abs{y}^2 \bigr) . \]
We restate the theorem we want to prove:
\begin{thm}
\label{principal}
Let $f$ be an even, bounded from below and continous function on $\R^d$. Then
\[ \int  \e ^ {-f} \, d\gamma_d  \int \e ^ { H f } \, d\gamma_d 
  \leq 1 . \]
\end{thm}
\section{A ``small $\epsilon$'' inequality}
\begin{Def}
Let $C>0$ and $\epsilon\in (0,1]$. Let $\mathcal{F}(C,\epsilon)$ be the class of functions on $\R^d$ satisfying the following properties:
\begin{itemize}
\item[(i)] $f$ is lipschitz with constant $C \epsilon $
\item[(ii)] for any $x$ and $h$ in $\R^d$:
\[ f(x+h) + f(x-h) - 2 f(x) \leq C \epsilon^2 \abs{h}^2 . \]
\end{itemize}
\end{Def}
This class $\mathcal{F}(C,\epsilon)$ is stable under various operations. For example it is clearly a convex set, and, if $f$ is in $\mathcal{F}(C,\epsilon)$ and $u\in \R^d$ then $x\mapsto f(x+u)$ and $x\mapsto f(-x+u)$ are also in $\mathcal{F}(C,\epsilon)$. Besides, if $(f_i)_{i\in I}$ is a family of functions in $\mathcal{F}(C,\epsilon)$ and if for some $x_0$, $\inf_{i\in I} f_i(x_0) > -\infty$, then $f=\inf_i f_i$ is still in $\mathcal{F}(C,\epsilon)$. Indeed, it is then clear that $f$ is nowhere equal to $-\infty$ and that it is $C\epsilon$-lipschitz. Let $x,h \in \R^d$ and let $i_0$ satisfy $f_{i_0} (x) \leq f(x) + \eta$ ; we have
\[ f(x+h) + f(x-h) - 2 f(x) \leq f_{i_0} (x+h) + f_{i_0}(x-h) -2 f_{i_0} (x)
                                    + 2\eta . \]
Applying property (ii) for $f_{i_0}$ and letting $\eta$ go to $0$, we get the result.  \\
Let $X$ be a standard Gaussian vector on $\R^d$, the expectation of a random variable will be denoted by $\E$.
\begin{lem}
\label{petitepsilon}
Let $C>0$, for every $\epsilon \in (0,1]$, for every \emph{even} function $f \in \mathcal{F}(C,\epsilon)$ we have
\[ \E \e^{-f (X) }  \E \e^{ H f (X) }
   \leq 1 +  K  \epsilon^3 , \]
with a constant $K$ \emph{depending on $C$ solely}.
\end{lem}
The proof of this lemma is based on a Taylor expansion and on a symmetric Poincar\'e inequality.
\begin{lem}
\label{sym-poincare}
Let $f$ be a smooth function in $L_2(\R^d, \gamma_d)$. Assume that $f$ is orthogonal to constants and to linear functions:
\begin{equation}
\label{poincare}
 \E f(X) = 0 \quad \text{and} \quad \E X f(X) = 0 . 
\end{equation}
Then 
\begin{equation}
 \E f(X)^2  \leq  \tfrac{1}{2}  \E \abs{\nabla f(X)}^2 . 
\end{equation}
\end{lem}
It is a well-known result, we recall its proof for completeness, the reader can also have a look at \cite{B-conj} where a more general statement is proved.
\begin{proof} We consider the sequence $(H_\alpha)_{\alpha \in \N^d}$ of the Hermite polynomials on $\R^d$. It is an orthonormal basis of $L_2(\gamma_d)$ in which the (unbounded) operator 
\[ L\colon f\mapsto \Delta f - x \cdot \nabla f \] 
is diagonal: for all $\alpha$ we have $-LH_\alpha = \abs{\alpha} H_\alpha$, where $\abs{\alpha}=\sum \alpha_i$. We refer to \cite[chapter 2]{pisier} for details. Moreover, we have the following integration by parts formula:
\[ - \E Lf(X)  g(X)  = 
 \E  \nabla f(X)  \cdot  \nabla g(X)  . \]
Let us write $f = \sum f_\alpha H_\alpha$. We have $-Lf= \sum \abs{\alpha} f_\alpha H_\alpha$ and \eqref{poincare} implies that $f_\alpha = 0$ when $\abs{\alpha}=0$ or $1$. Therefore
\[ \E \abs{\nabla f(X)}^2 = - \E Lf(X)  f(X) =
   \sum_{\abs{\alpha}\geq 2} \abs{\alpha} f_\alpha^2 
                           \geq  2  \E f(X)^2  , \]
which concludes the proof.
\end{proof}
\begin{proof}[Proof of Lemma~\ref{petitepsilon}] 
In what follows, $K_1,K_2,\dotsc$ are constants that depend only on $C$. Let $\epsilon \in (0,1]$ and $f\in \mathcal{F}(C,\epsilon)$, even. Replace $f$ by its convolution by some even regular Dirac approximation $(\rho_n)$ (Gaussian for example). Clearly $\rho_n \ast f$ is even and belongs to $\mathcal{F}(C,\epsilon)$. On the other hand, since $f$ is lipschitz, $\rho_n\ast f \rightarrow f$ uniformly, which implies $ H (\rho_n\ast f) \rightarrow Hf $. Moreover, being lipschitz, $f$ is bounded by a multiple of $1+\abs{x}$, and the same holds for $Hf$. Then $\e^{-\rho_n\ast f}$ and $\e^{ H (\rho_n\ast f) }$ are both dominated by some function $x\rightarrow A \e^{B\abs{x}}$. By the dominated convergence theorem, we obtain
\[ \E \e^{-(\rho_n \ast f) (X) }  \E  \e^{ H (\rho_n \ast f) (X) }   \rightarrow 
 \E \e^{-f (X) }  \E  \e^{ H f (X) } , \]
when $n$ goes to infinity. Hence it is enough to prove the inequality when $f$ is $\mathcal{C}^2$. Moreover, the inequality does not change if we add a constant to $f$. Hence we can also assume that $\E f(X) = 0$. When $f$ is $\mathcal{C}^2$, property (ii) implies that $\mathrm{Hess} f (x) \leq C \epsilon^2 \mathrm{Id}$ for any $x$. Hence by Taylor's formula 
\begin{itemize}
\item[(ii')] for any $x$ and $h$, $f(x+h)
      \leq f(x)+ \nabla f(x) \cdot h + \tfrac{C}{2} \epsilon^2 \abs{h}^2$.
\end{itemize}
First we estimate $\E \e^{-f (X)}$. We write
\[ \E \e^{ -f (X) } = \sum_{k=0}^\infty \tfrac{ (-1)^k }{ k! }
       \E f (X)^k. \]
To estimate the moments of $f(X)$, we use the concentration property of the Gaussian measure: if $\phi$ is a real valued, mean $0$ and $1$-lipschitz function on $\R^d$, then $\phi(X)$ is closed to $0$ with high probability. More precisely, for any $t>0$, we have
\[ \mathrm{Prob} \,  \bigl\{ \abs{\phi(X)} \geq t \bigr\}  \leq  2 \e^{-t^2/2} . \]
We refer to \cite[theorem 2.5]{ledoux} for a proof of this statement. This yields that there exists a universal constant $M$ such that for any $p\geq 1$
\[ \Bigl( \E \abs{\phi(X)}^p \Bigr)^{1/p}  \leq  M \sqrt{p}  . \]
Since $f$ is $C\epsilon$-lipschitz, $\frac{f}{C\epsilon}$ is 1-lipschitz, and it has mean $0$ so the preceding inequality applies. Thus, for $k\geq 3$, we can bound
$\abs{ (-1)^k \E f(X)^k }$ from above by $(MC\sqrt{k})^k \epsilon^3$. We obtain
\begin{equation}
\label{un}
 \E  \e^{ -f (X) } \leq 1 + \tfrac{1}{2} \E  f ( X )^2 
 + K_1  \epsilon^3 .
\end{equation}
In the same way, we have
\begin{equation}
\label{aa}
\E  \e^{f(X)} \leq 1 + \tfrac{1}{2} \E  f(X)^2 + K_1  \epsilon^3
                \quad \mathrm{and} \quad
\E  \bigl\lvert  \e^{f(X)} -1 \bigr\rvert \leq K_2 \epsilon^2.
\end{equation}
We now deal with $H f$. Trying $y= - \nabla f( x )$ in the definition of $Hf$, we get 
\[ H f (x)  \leq f ( x - \nabla f( x ) ) + 
\tfrac{1}{2} \abs{\nabla f(x)} ^2 . \]
Applying (ii') to $x$ and $h=-\nabla f(x)$, we obtain
\[ H f (x)  \leq
    f ( x) - \abs{\nabla f(x)}^2 +
        \tfrac{1}{2} \abs{\nabla f(x)}^2 +
                \tfrac{C}{2} \epsilon^2 \abs{\nabla f(x)}^2. \]
Remark that since $f$ is $C\epsilon$-lipschitz, we have $\abs{\nabla f(x)} \leq C \epsilon$ for any $x$. We obtain
\[ H f (x)  \leq f(x) -
\tfrac{1}{2} \abs{\nabla f(x)}^2 + \tfrac{1}{2} C^3  \epsilon^4. \]
Taking the exponential, and using $\abs{\nabla f}\leq C\epsilon$ again, we get
\begin{align*}  
\e^{ H f(x)  } & \leq  \e^{ f(x) } 
      \bigl(1 - \tfrac{1}{2}  \abs{\nabla f(x)}^2 
     + K_3  \epsilon^4 \bigr)  \bigl( 1 + K_4  \epsilon^4  \bigr) \\
& \leq  \e^{f(x)} - \tfrac{1}{2} \abs{\nabla f(x)}^2 + \tfrac{1}{2} ( 1 - \e^{f(x)} )
\abs{\nabla f(x)}^2 + K_5  \e^{f(x)}  \epsilon^4.
\end{align*}
We take the expectation and we use \eqref{aa}, we obtain
\begin{equation} 
\label{deux}
\E  \e^{ H f(X) }  \leq
      1 + \tfrac{1}{2} \E f(X)^{2}
           - \tfrac{1}{2} \E \abs{\nabla f(X)}^2 + K_6  \epsilon^3  .
\end{equation}
Multiplying \eqref{un} and \eqref{deux} we get
\begin{equation}
\label{trois}
\E  \e ^ { -f (X) }   \E  \e ^ {H f (X)} 
     \leq  1 + \E f (X)^2 - \tfrac{1}{2}  
     \E  \abs{\nabla f(X)}^ 2 + K \epsilon^3. 
\end{equation}
Now we use Lemma~\ref{sym-poincare}: since $f$ has mean $0$ and is even, it satisfies \eqref{poincare}, hence $\E f(X)^2 \leq \tfrac{1}{2} \E \abs{\nabla f(X)}^2 $, which concludes the proof.
\end{proof}
\begin{Rem}
Without the assumption ``$f$ even'' the Poincar\'e inequality would only say
\[ \E  f(X)^2   \leq  
       \E  \abs{\nabla f(X)}^2 , \]
and we would be able to prove
\[ \E  \e^{- f (X) }  
    \E  \e^{ (f \boxdot 
     \tfrac{c}{4}) (X) }  \leq 1 + K  \epsilon^3 . \]
\end{Rem}
In order to prove the symmetric property~($\tau$), we need to tensorize Lemma~\ref{petitepsilon}. This requires various technical tools, most of which are well known to specialists. We begin with a very simple property of the class $\mathcal{F}(C,\epsilon)$.
\begin{lem}
\label{detail2}
Let $E_1$ and $E_2$ be Euclidean spaces, and $\mu$ a probability measure on $E_1$. Let $f\colon E_1\times E_2 \rightarrow \R$ be such that for any $x\in E_1$ the
function $y\mapsto f(x,y)$ is in $\mathcal{F}(C,\epsilon)$. We define $\phi$ by
\[ \e^{-\phi(y)} = \int_{E_1} \e^{-f(x,y)} \, d\mu(x) . \]
Then, unless $\phi = -\infty$, $\phi$ belongs to $\mathcal{F}(C,\epsilon)$.
\end{lem}
\begin{proof} Property (i) is simple. We check (ii): multiplying by $-\tfrac{1}{2}$ the inequality $f(x,y+h)+f(x,y-h) \leq 2f(x,y) + C \epsilon^2 \abs{h}^2$, taking the exponential and integrating with respect to $\mu$, we obtain
\[ \int_{E_1} \e^{-f(x,y)} \, d\mu(x)  \e^{ -\tfrac{C\epsilon^2}{2} \abs{h}^2 } 
   \leq \int_{E_1} \e^{-\tfrac{1}{2} f(x,y+h)} 
           \e^{-\tfrac{1}{2} f(x,y-h)}\, d\mu(x) .\]
Then we use Cauchy-Schwarz to bound the right hand side, we take the log and we end up with the desired inequality for $\phi$.
\end{proof}
\section{Symmetrisation}
An important tool in geometry is Steiner's symmetrisation, see \cite{berger} for definition, properties and applications. In particular, it is the main idea behind Meyer and Pajor's proof of the Blashke-Santal\'o inequality. It is thus natural, and as far as we know this idea dates back to \cite{akm} to introduce functional analogues of this symmetrisation.
\begin{Def}
Let $E$ be a Euclidean space and $f\colon E \rightarrow \R$ be continuous. We define
\[  S f\colon x\mapsto \inf_{u\in E}  \Bigl( 
           \tfrac{1}{2} \bigl( f(u+x)+f(u-x) \bigr) + 
  \tfrac{1}{2} \abs{u}^2   \Bigr) . \]
The function $S f$ is even. \\
Let $E_1,E_2$ be Euclidean spaces and $g\colon E_1\times E_2\rightarrow \R$. We define $S_1 g$ and $S_2 g$ to be the symmetrisations of $g$ with respect to the first and the second variable, respectively:
\[  S_1 g\colon x,y \mapsto \inf_{u\in E_1}  \Bigl( 
           \tfrac{1}{2} \bigl( f(u+x,y)+f(u-x,y) \bigr) + \tfrac{1}{2} \abs{u}^2 
                 \Bigr) , \]
and similarly for $S_2$.
\end{Def}
\begin{lem}
\label{detail3}
Let $\gamma$ be the normal distribution on $E$. For any continuous $f$ on $E$, we have
\begin{equation}
\label{alpha}
 \int_E \e^{ -f } d\gamma \leq \int_E \e^{ - Sf } d\gamma .
\end{equation}
Let $\gamma_1,\gamma_2$ be the normal distributions on $E_1$ and $E_2$, respectively. Any continuous $g$ on $E_1\times E_2$ satisfies
\begin{equation}
\label{beta}
\int_{E_1\times E_2} \e^{ H g } d\gamma_1 \otimes d\gamma_2 
\leq \int_{E_1\times E_2} \e^{ H S_2 g } d\gamma_1 \otimes d\gamma_2 .
\end{equation}
\end{lem}
\begin{proof} Set $\tilde {f}\colon x\rightarrow f(-x)$. By definition of $Sf$, we have for every $z$ and $u\in\R^d$
\begin{align*}
Sf(z) + \tfrac{1}{2} \abs{z}^2 & \leq 
      \tfrac{1}{2} f (u+z) + \tfrac{1}{2} f (-z+u) +
        \tfrac{1}{2} \abs{u}^2 + \tfrac{1}{2} \abs{z}^2 \\
 & =  \tfrac{1}{2} \bigl( f(u+z)+ \tfrac{1}{2} \abs{u+z}^2 ) +
         \tfrac{1}{2} \bigl( \tilde{f} (z-u) + \tfrac{1}{2} \abs{z-u}^2 \bigr) .
\end{align*}
Let $x,y\in \R^d$, we apply this inequality to $z=\frac{x+y}{2}$ and $u=\frac{x-y}{2}$. We obtain
\[ Sf( \frac{x+y}{2} ) + \tfrac{1}{2} \bigl\lvert \frac{x+y}{2} \bigr\rvert^2 \leq
       \tfrac{1}{2} \bigl( f(x)+ \tfrac{1}{2} \abs{x}^2 ) +
        \tfrac{1}{2} \bigl( \tilde{f} (y) + \tfrac{1}{2} \abs{y}^2 \bigr) . \]
Applying the Pr\'ekopa-Leindler inequality we get \eqref{alpha}. \\
Inequality \eqref{beta} is the functional version of the following: if $K$ is a centrally symmetric convex body and $K_u$ its Steiner symmetral with respect to some direction $u$ then $\lvert  (K_u)^\circ \rvert \geq \lvert  K^\circ \rvert$. This is the key argument in Meyer and Pajor's proof of the symmetric Santal\'o inequality, see \cite{meyer-pajor}. Actually this idea dates back to Saint-Raymond \cite{saint-raymond} although he considered a symmetrisation of his own instead of Steiner's. Let us prove \eqref{beta}. It follows from
\begin{equation}
\label{steiner}
-S_1(-Hg) \leq H S_2 g .
\end{equation}
Indeed, combining it with \eqref{alpha} we get for every $y\in E_2$
\begin{align*} 
\int \e^{ H g (x,y)} \, d\gamma_1 (x) & = 
      \int \e^{-(-Hg)(x,y)} \, d\gamma_1 (x) \\
   & \leq  \int \e^{ - S_1(- H g(x,y) ) } \, d\gamma_1 (x)
     \leq  \int \e^{ H S_2 g(x,y) } \, d\gamma_1 (x) , 
\end{align*}
which implies \eqref{beta}. Proof of \eqref{steiner}: notations are heavy but it is straightforward. 
\[ \begin{split}
2 H S_2 g (x,y) 
       & =  \inf_{u,v}  \bigl\{  2 S_2 g(x+u,y+v) +
              \abs{u}^2 + \abs{v}^2  \bigr\}    \\
       &  = \inf_{u,v,w} \bigl\{  g(x+u,w+y+v)+g(x+u,w-y-v) \\ 
       & \quad \qquad  + \abs{u}^2+ \abs{v}^2+ \abs{w}^2 \bigr\}  .
\end{split}  \]
Since $g$ is even, the latest is the same as   
\[ \inf_{u,v,w}  \bigl\{  g(x+u,w+y+v)+g(-x-u,-w+y+v)+
            \abs{u}^2+ \abs{v}^2+ \abs{w}^2  \bigr\} . \]
Whereas
\begin{align*}
-2 S_1 (-H g) (x,y) & =   - \inf_{t}  \bigl\{  (-Hg)(t+x,y)+
               (-Hg)(t-x,y)+\abs{t}^2  \bigr\} \\
       & =  \sup_{t}  \bigl\{ Hg(t+x,y)+ Hg(t-x,y)-\abs{t}^2  \bigr\} .
\end{align*}
Hence we have to prove that for any $t,u,v,w$
\[ \begin{split}
Hg (t+x,y) + Hg (t-x,y) & \leq  g (x+u,w+y+v) + g (-x-u,-w+y+v) \\
   &  + \abs{u}^2 + 
    \abs{v}^2 + \abs{w}^2 + \abs{t}^2 ,
\end{split} \]
which is clear: use twice $Hg(a)- g(b) \leq \frac{1}{2} \abs{b-a}^2 $.
\end{proof}
\section{Tensorisation}
\label{sectens}
\begin{Def}
Let $n \in \N^*$, we define the class $\mathcal{F}_n (C,\epsilon)$ by induction on $n$: let $f$ be a function on $(\R^d)^n$, we say that $f$ belongs to $\mathcal{F}_n (C,\epsilon)$ if 
\begin{itemize}
\item[-] for every $y \in \R^d$, the function $x\in (\R^d)^{n-1} \mapsto f(x,y)$ is in $\mathcal{F}_{n-1}(C,\epsilon)$
\item[-] for every $x\in (\R^d)^{n-1}$, the function $y \in \R^d \mapsto f(x,y)$ is in $\mathcal{F}(C,\epsilon)$.
\end{itemize}
\end{Def}
In other words $\mathcal{F}_n (C,\epsilon)$ is the class of functions on $\R^d\times\dotsb \times \R^d$ that belong to $\mathcal{F}(C,\epsilon)$ with respect to each coordinate separately. The crucial point is that the class $\mathcal{F}_n (C,\epsilon)$ is stable under symmetrisation.
\begin{lem}
\label{detail1}
Let $f$ belong to $\mathcal{F}_n (C,\epsilon)$, set $E_1 = (\R^d)^{n-1}$ and $E_2 = \R^d$. Then $S_2 f$ also belongs to $\mathcal{F}_n (C,\epsilon)$.
\end{lem}
\begin{proof} It follows from the stability properties of the class $\mathcal{F} (C,\epsilon)$ that we mentionned earlier: for every $x \in E_1$ and $u\in E_2$, the function
\[ y \mapsto \tfrac{1}{2} 
   \bigl( f(x,u+y)+f(x,u-y) \bigr) + \tfrac{1}{2} \abs{u}^2  \]
is in $\mathcal{F} (C,\epsilon)$ so the same is true for the infimum over $u$, provided that this infimum is not $-\infty$, which is clear, since $\tfrac{1}{2} \bigl( f(x,u+y)+f(x,u-y) \bigr)$ is lipschitz in $u$. Similarly, for any $y,u \in E_2$, the function
\[ x \mapsto  \tfrac{1}{2} 
   \bigl( f(x,u+y)+f(x,u-y) \bigr) + \tfrac{1}{2} \abs{u}^2  \]
is in $\mathcal{F}_{n-1} (C,\epsilon)$ and the same is true for the infimum over $u$.
\end{proof} 
The next lemma is a tensorisation of Lemma~\ref{petitepsilon}, with the same constant $K$. 
\begin{lem}
\label{tensorisation}
Let $X_1,\dotsc,X_n$ be independent copies of $X$ and let
$f$ be \emph{even} and in $\mathcal{F}_n (C,\epsilon)$. Then
\[ \E  \e^{ - f( X_1,\dotsc,X_n ) } 
  \E  \e^{ H f(X_1,\dotsc,X_n) } 
  \leq  ( 1 + K \epsilon^3) ^n . \]
\end{lem}
\begin{proof} It is done by induction. When $n=1$, it is Lemma~\ref{petitepsilon}. Let $n\geq 2$, we assume that the resluts holds for $n-1$. Let $f\colon (\R^d)^n \rightarrow \R$ be even and in $\mathcal{F}_n (C,\epsilon)$. Set $E_1 = (\R^d)^{n-1}$, $E_2=\R^d$ and $\tilde{X}=(X_1,\dotsc,X_{n-1})$, it is a normal vector on $E_1$. We have to prove
\[  \E  \e^{ - f( \tilde{X},X_n ) }  \E  \e^{Hf(\tilde{X},X_n)} 
                     \leq  (1+K\epsilon)^n . \]
Let $g = S_2 f$. We use Lemma~\ref{detail3}, for every $x\in E_1$ we have: $\E   \e^{- f (x,X_n)} \leq \E  \e^{-S_2 f (x,X_n) }$, which implies that
\[ \E  \e^{-f ( \tilde{X},X_n ) }  \leq 
              \E  \e^{ - S_2 f (\tilde{X},X_n) } . \]
On the other hand, the second part of Lemma~\ref{detail3} gives
\[ \E  \e^{H f (\tilde{X},X_n)}  \leq 
             \E  \e^{ H S_2 f (\tilde{X},X_n) } . \]
Hence it is enough to prove that
\begin{equation}
\label{aaa}  
\E  \e^{ - g(\tilde{X},X_n) }  \E  \e^{ H g (\tilde{X},X_n) } 
            \leq  (1+K\epsilon)^n . 
\end{equation}
By Lemma~\ref{detail1}, we have $g\in \mathcal{F}_n(C,\epsilon)$. By definition of $S_2$, for every $x\in E_1$, the function $y\mapsto g(x,y)$ is even. But as $g$ is globally even we have also $g(-x,y)=g(x,-y)=g(x,y)$. Hence, for any $y\in E_2 $, the function $g^y\colon x\rightarrow g(x,y)$ is even and in $\mathcal{F}_{n-1} (C,\epsilon)$. By the induction assumption 
\begin{equation}
\label{hypind}
 \E  \e^{- (g^y) ( \tilde{X} ) }
 \E  \e^{H  (g^y) ( \tilde{X} ) } 
  \leq  (1+ K  \epsilon^3)^{n-1} . 
\end{equation}
The following computation can be found in \cite{propriete-tau}, we recall it for the sake of completeness. We define the operator $H_1$ by 
\[ H_1 g(x,y) = H (g^y) (x) = \inf_{u\in E_1} \bigl( g(x+u,y) +
                                \tfrac{1}{2} \abs{u}^2 \bigr) . \]
We define a new function $\phi$ by
\[ \e^{-\phi(y)}  =  \E  \e^{- g ( \tilde{X} , y ) } . \]
For any $x\in E_1$, the function $y \rightarrow g(x,y)$ is even and belongs to $\mathcal{F}(C,\epsilon)$. By Lemma~\ref{detail2}, the same is true for $\phi$. Hence we can apply Lemma~\ref{petitepsilon} to get
\begin{equation}
\label{bbb}
 \E  \e^{ - \phi ( X_n ) } 
   \E  \e^{ H \phi ( X_n ) } 
    \leq   1 + K  \epsilon^3  .
\end{equation}
On the other hand, inequality \eqref{hypind} implies that for any $y$ and $u$ in $\R^d$ we have
\begin{equation}
\label{ccc}
 \E  \e^{ H_1 g ( \tilde{X} , y+u ) 
         + \frac{1}{2} \abs{u}^2 }  \leq 
\e^{\phi(y+u)+\frac{1}{2} \abs{u}^2 }  ( 1+ K  \epsilon^3 )^{n-1} .
\end{equation}
Let $(x,y)\in E_1\times E_2$. For any $u \in \R^d$ we have $H_1 g ( x,y+u ) + \frac{1}{2} \abs{u}^2 \geq H g (x,y)$. Hence, taking the infimum over $u$ in \eqref{ccc}, we get
\begin{equation}
\label{etape1}
 \E  \e^{ H g ( \tilde{X},y) }
        \leq  \e^{ H \phi (y) } 
 \bigl( 1+ K  \epsilon^3 \bigr)^{n-1} ,
\end{equation}
which of course implies that
\begin{equation}
\label{ddd}
\E  \e^{ H  g ( \tilde{X}, X_n) }  \leq  \E  \e^{ H \phi (X_n) } 
        \bigl( 1+ K  \epsilon^3 \bigr) ^{n-1} . 
\end{equation}
Combining \eqref{ddd} with \eqref{bbb} yields \eqref{aaa} and the proof is complete.
\end{proof}
\section{Proof of Theorem~\ref{principal}}
Let $n\in \N$, and $f$ be an even function from $\R^d$ to $\R$ belonging to $\mathcal{F}(C,1)$. Set
\[
g\colon 
 (x_1,\dotsc,x_n) \in (\R^d)^n \mapsto  f \bigl(\tfrac{1}{\sqrt{n}} (x_1+\dotsb+x_n) \bigr).
\]
It is clear that $g$ is even and belongs to $\mathcal{F}_n( C , \frac{1}{ \sqrt{n} } )$. Applying Lemma~\ref{tensorisation}, we obtain
\begin{equation}
 \label{g-ineq}
 \E \e^{ - g ( X_1,\dotsc,X_n ) } 
  \E  \e^{ H g (X_1,\dotsc,X_n) } 
  \leq  ( 1 + K  \tfrac{ 1 }{ n^{3/2} }  )^n \leq
                                 1 + \tfrac{K'}{\sqrt{n}}. 
\end{equation}
The convexity of the square of the norm implies that
\begin{align*}
 H g (x_1,\dotsc,x_n)  
    & =   \inf_{u_1,\dotsc,u_n}   
 f \bigl(  \tfrac{1}{ \sqrt{n} } {\textstyle \sum} (x_i+u_i) \bigr)
       +   \tfrac{1}{2}  {\textstyle \sum} \abs{u_i}^2    \\
 & \geq    \inf_{u_1,\dotsc,u_n}   
 f \bigl(  \tfrac{1}{ \sqrt{n} } {\textstyle \sum} x_i 
       + \tfrac{1}{ \sqrt{n} } {\textstyle \sum} u_i  \bigr)  
+ \tfrac{1}{2}  \bigl\lvert \tfrac{1}{ \sqrt{n} }  {\textstyle \sum} u_i \bigr\rvert^2, \\
 & = H f \bigl( \tfrac{1}{ \sqrt{n} } {\textstyle \sum} x_i \bigr) .
\end{align*} 
We combine this inequality with \eqref{g-ineq}, since 
the random vector $\frac{1}{\sqrt{n}}\sum X_i$ has the same law as $X$, we obtain
\[ \E  \e^{ - f( X ) }  \E  \e^{ H f ( X ) }  \leq 
    1 + \tfrac{ K' }{ \sqrt{n} }  . \]
Now we let $n\rightarrow \infty$ and we get the result for $f$. Hence we have the inequality for any even $f$ in $\mathcal{F}(C,1)$.
But this holds for any $C$, so we can take $C$ as big as we want. In particular, we get the result for any even and $\mathcal{C}^2$ function
with compact support. And a density argument shows that the inequality is valid for any function that is even, continuous and bounded from below.
\section*{Remarks}
It is also possible to derive the usual property~($\tau$) for the Gaussian measure from the Poincar\'e inequality. The function $f$ is not
assumed to be even anymore, we use the non-symmetric version of Lemma~\ref{petitepsilon} and then perform the tensorisation
(just as in Section~\ref{sectens} but without the symmetrisation). Of course, this is a bit more complicated than the usual proof (using the Pr\'ekopa-Leindler inequality) but we find it interesting to see that property~($\tau$) has to
do with the first eigenvalue of the Laplacian whereas its symmetric version has to do with the second eigenvalue.  \\
Unfortunately we are not able to go farther: what happens if $f$ is orthogonal to constants, linear functionals and polynomials of degree $2$?
Actually even in the case of a function that is orthogonal to constants and linear functionals our method does not give the result: we had to
suppose that our function was even, which is stronger. \\
Lastly, following Klartag \cite{klartag}, we explain how to extend the symmetric property~($\tau$) to measures that are
even and log-concave with respect to the Gaussian. It is well known that property~($\tau$) is stable under $1$-Lipschitz image. Clearly, the symmetric property~($\tau$) will have the following stable under pushforward by a $1$-Lipschitz odd map. For example, 
it was proved by Caffarelli (see \cite{caffarelli}) that if a probability measure $\mu$ is log-concave with respect to the Gaussian measure -- meaning that $d\mu = \e^{-V} d \gamma$ for some $V$ convex -- then the Brenier map transporting the Gaussian measure to $\mu$ is $1$-Lipschitz. Now suppose additionally that $\mu$ is even, then the Brenier map is automatically odd so $\mu$ is the pushfoward of the Gaussian measure by a $1$-Lipschitz odd map. Therefore, probability measures that are even and log-concave with respect 
to the Gaussian measure, in particular restrictions of the Gaussian measure to symmetric convex bodies, satisfy the symmetric property~($\tau$).
\providecommand{\bysame}{\leavevmode\hbox to3em{\hrulefill}\thinspace}


\begin{thebibliography}{10}

\bibitem{akm}
S.~Artstein, B.~Klartag, and V.~Milman, \emph{The {S}antal\'o point of a
  function, and a functional form of {S}antal\'o inequality}, Mathematika
  \textbf{51} (2005), 33--48.

\bibitem{these-ball}
K.~Ball, \emph{Isometric problems in $\ell_p$ and sections of convex sets},
  Ph{D} dissertation, {U}niversity of {C}ambridge, 1986.

\bibitem{ball}
K.~Ball, \emph{An elementary introduction to modern convex geometry}, in Flavors
  of geometry, edited by S.~Levy, Cambridge {U}niversity {P}ress, 1997.

\bibitem{berger}
M.~Berger, \emph{Geometry}, vol. I-II, translated from the {F}rench by {M}.~{C}ole and {S}.~{L}evy, Universitext,  Springer, 1987.

\bibitem{caffarelli}
L.A. Caffarelli, \emph{Monotonicity properties of optimal transportation and
  the {FKG} and related inequalities}, Comm. math. phys. \textbf{214} (2000),
  no.~3, 547--563.

\bibitem{B-conj}
D.~Cordero-Erausquin, M.~Fradelizi, and B.~Maurey, \emph{The ({B}) conjecture
  for the {G}aussian measure of dilates of symmetric convex sets and related
  problems}, J. {F}unct. {A}nal. \textbf{214} (2004), 410--427.

\bibitem{fm}
M.~Fradelizi and M.~Meyer, \emph{Some functional forms of {B}laschke-{S}antal\'o inequality},
Math. Z., \textbf{256} (2007), no.~2, 379--395.

\bibitem{klartag}
B.~Klartag, \emph{Marginals of geometric inequalities}, Geometric Aspects of
  Functional Analysis, Lecture {N}otes in {M}ath. 1910, Springer, 2007,
  pp.~133--166.

\bibitem{ledoux}
M.~Ledoux, \emph{The concentration of measure phenomenon}, Mathematical {S}urveys and Monographs, American Mathematical Society, 2001.

\bibitem{lz}
E.~Lutwak and G.~Zhang, \emph{Blaschke-{S}antal\'o inequalities}, J. Diff.
  Geom. \textbf{47} (1997), no.~1, 1--16.

\bibitem{propriete-tau}
B.~Maurey, \emph{Some deviation inequalities}, Geom. {F}unct. {A}nal.
  \textbf{1} (1991), no.~2, 188--197.

\bibitem{meyer-pajor}
M.~Meyer and A.~Pajor, \emph{On {S}antal\'o's inequality}, Geometric {A}spects
  of {F}unctional {A}nalysis, Lecture {N}otes in {M}ath. 1376, Springer, 1989,
  pp.~261--263.

\bibitem{pisier}
G.~Pisier, \emph{The volume of convex bodies and {B}anach space geometry},
  Cambridge Tracts in Mathematics, Cambridge {U}niversity {P}ress, 1989.

\bibitem{saint-raymond}
J.~Saint-Raymond, \emph{Sur le volume des corps convexes sym\'etriques},
  S\'eminaire d'initiation \`a l'Analyse \textbf{11} (1980-81).

\bibitem{santalo}
L.A. Santal\'o, \emph{Un invariante afin para los cuerpos convexos del espacio
  de $n$ dimensiones}, Portugaliae {M}ath. \textbf{8} (1949), 155--161.

\end{thebibliography}
\end{document}